\documentclass[oneside,english]{amsart}

\newcommand\version{October 29, 2026}
\usepackage[top=2.5cm, bottom=3.5cm, right=2.2cm, left=2.2cm]{geometry}

\usepackage{color}

\usepackage[T1]{fontenc}
\usepackage[latin9]{inputenc}
\usepackage{amsbsy}
\usepackage{amstext}
\usepackage{amsthm}
\usepackage{amssymb}
\usepackage{mathtools}
\usepackage{enumerate}
\usepackage[colorlinks=true]{hyperref}
\hypersetup{
	colorlinks=true,%
	citecolor=red,%
	filecolor=black,%
	linkcolor=blue,%
	urlcolor=blue
}

\makeatletter
\newtheorem{thm}{Theorem}[section]
\newtheorem{lem}[thm]{Lemma}
\newtheorem{cor}[thm]{Corollary}

\newtheorem{remark}[thm]{Remark}

\providecommand{\definitionname}{Definition}
\theoremstyle{plain}

\newcommand{\eq}[1]{\begin{equation}\allowdisplaybreaks\begin{alignedat}{2} #1 \end{alignedat}\end{equation}}

  \newcommand{\R}{\mathbb{R}}

   \newcommand{\sech}{\operatorname{sech}}
   \newcommand{\artanh}{\operatorname{artanh}}

\def\<{\langle}
\def\>{\rangle}

\def \intsn {\int_{\mathbb{S}^n}}
\def \gSn{g_{\mathbb{S}^n}}
\def\triangle {\Delta}

\newcommand{\sgn}{\operatorname{sgn}}

\def\circledwedge{\setbox0=\hbox{$\bigcirc$}\relax \mathbin {\hbox
to0pt{\raise.5pt\hbox to\wd0{\hfil $\wedge$\hfil}\hss}\box0 }}
 
\makeatother

\usepackage{babel}

\allowdisplaybreaks

\numberwithin{equation} {section}

\begin{document}

\title[Q-R Sobolev inequality]{A new sharp Sobolev inequality between the $Q$-curvature\\ and the scalar curvature on Euclidean space}

\author{Rupert L. Frank}
\address[R. L. Frank]{Mathematisches Institut, Ludwig-Maximilians Universit\"at M\"unchen, Theresienstr. 39 \\80333 M\"unchen, Germany\\
and Munich Center for Quantum Science and Technology, Schellingstr. 4 \\80799 M\"unchen, Germany}
\email{r.frank@lmu.de}

\author{Yuxin Ge}
\address[Y.~Ge]{Institut de Math\'ematiques de Toulouse\\
Universit\'e Paul Sabatier\\
118, route de Narbonne\\
31062 Toulouse Cedex, France}
\email{yge@math.univ-toulouse.fr}

\author{Guofang Wang}%\thanks{The second named author is partly supported by SFB/TR71 of DFG}
\address[G.~Wang]{Albert-Ludwigs-Universit\"at Freiburg,
Mathematisches Institut,
Ernst-Zermelo-Str. 1,
D-79104 Freiburg, Germany}
\email{guofang.wang@math.uni-freiburg.de}

\author{Wei Wei}
\address[W.~Wei]{Nanjing University, School of Mathematics, Nanjing 210093, P.R. China}
\email{wei\_wei@nju.edu.cn}

\begin{abstract}
We establish a sharp Sobolev-type inequality on $\mathbb{R}^n$, $n\geq 5$, that interpolates between the second-order energy and a weighted first-order term. Specifically, we determine the sharp constant $S_{4,2}$ in the inequality
\begin{equation*}
\int_{\mathbb R^n} (\Delta u)^2\,dx
\ge
S_{4,2} \left(\int_{\mathbb{R}^n}|\nabla u|^{2}|u|^{\frac{4}{n-4}}\,dx\right)^{\frac{n-4}{n-2}},
\end{equation*}
for any $u\in\dot H^2(\mathbb{R}^n)$, and we classify the cases of equality. 
This solves the problem of minimizing the total $Q$-curvature of among metrics on the sphere, conformal to the standard metric, with fixed total scalar curvature.

\end{abstract}

\date{\version}

\subjclass[2020]{Primary 53C21; Secondary 58J60, 46E35, 35J30.}
\keywords{$Q$-curvature, scalar curvature, sharp Sobolev inequality, conformal geometry, concentration--compactness, rearrangement.}

\maketitle

%%%%%%%%%%%%%%%%%%%%%%%%%%%%%%%%%
%%%%%%%%%%%%%%%%%%%%%%%%%%%%%%%%%

\section{Introduction}

The classical Sobolev inequalities state that
\[
\int_{\mathbb R^n}|\nabla u|^2\,dx\ge S_{2,0}\left(\int_{\mathbb R^n} |u|^{\frac{2n}{n-2}}\,dx\right)^{\frac{n-2}{n}},
\qquad
\int_{\mathbb R^n}|\Delta u|^2\,dx\ge S_{4,0}\left(\int_{\mathbb R^n} |u|^{\frac{2n}{n-4}}\,dx\right)^{\frac{n-4}{n}} \,.
\]
Here we use the convention that $S_{2,0}$ and $S_{4,0}$ denote the (known) optimal constants in these inequalities. We refer to~\cite{Franknotes} for references and background on sharp Sobolev inequalities. Via stereographic projection, these inequalities are equivalent to sharp conformal inequalities on the sphere: for every metric $g\in[g_{\mathbb S^n}]$, the conformal class of the canonical round metric, we have
\begin{equation}\label{ineq:S1}
\frac{\intsn R_g\,dv_g}{\bigl(\mathrm{vol}(g)\bigr)^{\frac{n-2}{n}}}
\ge
\frac{\intsn R_{\gSn}\,dv_{\gSn}}{\bigl(\mathrm{vol}(\gSn)\bigr)^{\frac{n-2}{n}}}
\end{equation}
and
\begin{equation}\label{ineq:S2}
\frac{\intsn Q_g\,dv_g}{\bigl(\mathrm{vol}(g)\bigr)^{\frac{n-4}{n}}}
\ge
\frac{\intsn Q_{\gSn}\,dv_{\gSn}}{\bigl(\mathrm{vol}(\gSn)\bigr)^{\frac{n-4}{n}}}.
\end{equation}
Here $R_g$ is the scalar curvature of $g$ and $Q_g$ is its $Q$-curvature, introduced by Branson in \cite{Branson1985}. The $Q$-curvature is a higher-order analogue of the scalar curvature. 

On $\mathbb S^n$ the scalar and $Q$-curvatures can both be described in a unified way in terms of the GJMS operators: for $k=1,2$,
\[
Q_{2k}\coloneqq\frac{2}{n-2k}L_{2k}^g(1),
\qquad
L_{2k}\coloneqq\prod_{j=1}^k\left(-\Delta+\frac{(n-2j)(n+2j-2)}{4}\right).
\]

A different, nonlinear generalization of scalar curvature is provided by the $\sigma_k$-curvatures; in particular, $\sigma_1$ agrees with the scalar curvature up to a constant factor. Sharp Sobolev-type inequalities comparing $\sigma_k$ and $\sigma_\ell$ ($k>\ell\ge0$) were established in~\cite{GVW,GWDuke,GeWangWei,GWCAG}. %(\textcolor{red}{Please add others, do not miss, for example, Chang-Yang, Brendle-Via})

The inequality \eqref{ineq:S2} compares total $Q$-curvature with volume. It is therefore natural to ask for sharp comparisons between different total curvatures within a fixed conformal class. Motivated by the sharp inequality between total $\sigma_2$ and total $\sigma_1$, Ge--Wang--Wei~\cite{GeWangWei} recently proved that if $g\in[g_{\mathbb S^n}]$ satisfies $R_g>0$, then
\begin{equation}\label{SharpbetweenQandR}
\frac{\intsn Q_g\,dv_g}{\bigl(\intsn R_g\,dv_g\bigr)^{\frac{n-4}{n-2}}}
\ge
\frac{\intsn Q_{\gSn}\,dv_{\gSn}}{\bigl(\intsn R_{\gSn}\,dv_{\gSn}\bigr)^{\frac{n-4}{n-2}}}.
\end{equation}
The positivity assumption on $R_g$ restricts the scope of \eqref{SharpbetweenQandR} and, more importantly, prevents a direct translation into an equivalent sharp inequality on $\mathbb R^n$.

Our first result is that \eqref{SharpbetweenQandR} remains true without any sign condition on the scalar curvature.

\begin{thm}\label{maintheorem}
Let $n\ge 5$. For every metric $g$ conformal to the round metric $\gSn$,% with $\int_{\mathbb{S}^n} R_g dv_g>0$, 
we have
\[
\frac{\intsn Q_g\,dv_g}{\bigl(\intsn R_g\,dv_g\bigr)^{\frac{n-4}{n-2}}}
\ge
\frac{\intsn Q_{\gSn}\,dv_{\gSn}}{\bigl(\intsn R_{\gSn}\,dv_{\gSn}\bigr)^{\frac{n-4}{n-2}}}.
\]
Equality holds if and only if $g=c\,\Phi^{*}g_{\mathbb S^{n}}$ for some $c>0$ and $\Phi\in\operatorname{Conf}(\mathbb S^{n})$.
\end{thm}

Here $\operatorname{Conf}(\mathbb S^{n})$ denotes the set of conformal diffeomorphisms of $(\mathbb S^n,[\gSn])$, that is, the set of M\"obius transformations.

Note that the Yamabe constant on the round metric $\gSn$ is positive, so that for every metric $g$ conformal to the round metric $\gSn$,  $\int_{\mathbb{S}^n} R_g dv_g>0$. Moreover, Theorem \ref{maintheorem}, together with \eqref{ineq:S1}, implies \eqref{ineq:S2}.

For us, it will be more convenient to work on $\mathbb R^n$. Let $\dot H^2(\mathbb{R}^n)$ denote the completion of $C_0^{\infty}(\mathbb{R}^n)$ with respect to the norm $\|u\|_{\dot H^2(\mathbb{R}^n)}\coloneqq\left(\int_{\mathbb{R}^n}(\Delta u)^2\,dx\right)^{1/2}$ and define
\[
S_{4,2}\coloneqq\inf_{u\in \dot H^2(\mathbb{R}^n)\setminus\{0\}} E(u),
\qquad
E(u)\coloneqq\frac{\int_{\mathbb{R}^n}(\Delta u)^{2}\,dx}{\left(\int_{\mathbb{R}^n}|\nabla u|^{2}|u|^{\frac{4}{n-4}}\,dx\right)^{\frac{n-4}{n-2}}}.
\]
It is a simple consequence of the Sobolev embeddings $\dot H^2(\R^n)\subset \dot W^{1,2n/(n-2)}\subset L^{2n/(n-4)}(\R^n)$ and H\"older's inequality that $S_{4,2}>0$. Our goal is to determine the explicit value $S_{4,2}$ and characterize all functions $u$ with $E(u)=S_{4,2}$.

The following is our second main result.

\begin{thm}\label{mainthm2}
Let $n\ge 5$. Any function $u\in \dot H^2(\mathbb{R}^n)$ satisfies
\begin{equation}
\int_{\mathbb R^n} (\Delta u)^2\,dx
\ge
\left(\frac{n-2}{n-4}\right)^{\frac{2(n-4)}{n-2}}
\frac{\Gamma\left(\frac{n}{2}+2\right)}{\Gamma\left(\frac{n}{2}-2\right)}
\left[\frac{\Gamma\left(\frac{n}{2}+1\right)}{\Gamma\left(\frac{n}{2}-1\right)}\right]^{-\frac{n-4}{n-2}}
|\mathbb{S}^n|^{\frac{2}{n-2}}
\left(\int_{\mathbb{R}^n}|\nabla u|^{2}|u|^{\frac{4}{n-4}}\,dx\right)^{\frac{n-4}{n-2}}.
\end{equation}
Equality holds if and only if there exist $c\in\mathbb{R}$, $\lambda>0$, and $x_0\in\mathbb{R}^n$ such that
\[
 u(x)=c\left(\frac{2\lambda}{\lambda^2+|x-x_0|^2}\right)^{\frac{n-4}{2}}.
\]
\end{thm}

By stereographic projection and conformal covariance, Theorem~\ref{maintheorem} is a direct consequence of Theorem \ref{mainthm2} for positive smooth functions $u$. In fact, the two theorems are equivalent, since the smoothness assumption can be relaxed to $\dot H^2(\R^n)$ by approximation and since the case of nonnegative functions implies the general case; see the proof of Lemma~\ref{lem:radial-minimizer} below.

The proof proceeds as follows. We prove the existence of minimizers via a concentration compactness argument. We then obtain radial symmetry by symmetric decreasing rearrangement and Talenti's theorem. Finally, adapting the approach of Chen--C.~Li--Ou~\cite{ChenLiOu}, we derive the explicit form of minimizers and classify all of them using a lemma of Y.~Li and Zhu~\cite{LiZhu}. We also provide an alternative ODE proof of the classification result.

During the preparation of  this paper, we learned that Han Lu~\cite{lu2026} proved a Liouville theorem for the $Q/R$ equation in $\R^n$ for positive solutions $u$ with $-\Delta u\ge0$. Our argument does not rely on such a Liouville theorem and may extend to sharp Sobolev inequalities comparing $Q_{2k}$ and $R_g$.

\medskip
\noindent{\it Acknowledgments.}
This work was carried out while W.~Wei was visiting the University of Freiburg, supported by an Alexander von Humboldt research fellowship. She thanks the Institute of Mathematics at the University of Freiburg for its hospitality. She is also partially supported by NSFC (Grant No.~12571218). R.L.F.~acknowledges partial support through the German Research Foundation grants EXC-2111-390814868 and TRR 352--Project-ID 470903074, and G.W. through IRTG 3132.%\texttt{Guofang: Graduiertenkolleg?}

%%%%%%%%%%%%%%%%%%%%%%%%%%%%%%%%

%%%%%%%%%%%%%%%%%%%

%%%%%%%%%%%%%%%%%%%

\section{Existence of minimizers}\label{sec:existence}

In this section we show that the minimization problem defining $S_{4,2}$ admits a minimizer. In fact, we prove more, namely that minimizing sequences are relatively compact up to symmetries.

\begin{thm}\label{main2}
	Let $(u_i)\subset \dot H^2(\mathbb{R}^n)$ be a sequence with
	\begin{equation}\label{eq:sequenceass}
		\int_{\mathbb{R}^n} (\Delta u_i)^2\,dx \to 1,
		\qquad
		\int_{\mathbb{R}^n} |u_i|^{\frac{4}{n-4}}|\nabla u_i|^2\,dx \to \left(S_{4,2}\right)^{-\frac{n-2}{n-4}}.
	\end{equation}
	Then, after passing to a subsequence and applying a translation and a dilation, $u_i$ converges strongly in $\dot H^2(\mathbb{R}^n)$. In particular, the infimum defining $S_{4,2}$ is attained.
\end{thm}

Our proof follows a simplification of Lions's concentration compactness method that was recently introduced by Dietze and Nam \cite{DietzeNam25}.

\begin{proof}
	\textit{Step 1.} We first show that, after translations and dilations, a subsequence
	converges weakly in $\dot H^2(\mathbb R^n)$ to a non-zero limit.
	
	Multiplying $u_i$ by positive constants tending to $(S_{4,2})^{1/2}$, we may
	assume that
	\begin{equation}
		\int_{\mathbb R^n}|u_i|^{\frac4{n-4}}|\nabla u_i|^2\,dx=1,
		\qquad
		\int_{\mathbb R^n}(\Delta u_i)^2\,dx\longrightarrow S_{4,2}.
		\label{eq:normal}
	\end{equation}
	Let $B_r(a)$ denote the ball of radius $r$ about $a$, and let $B_r=B_r(0)$.
	Choose $\delta>0$ so small that $B_2$ can be covered by fewer than
	$1/(2\delta)$ unit balls. Using translations and dilations, which preserve
	both integrals in \eqref{eq:normal}, we may further assume that
	\begin{equation}
		\int_{B_1}|u_i|^{\frac4{n-4}}|\nabla u_i|^2\,dx
		=\sup_{a\in\mathbb R^n}\int_{B_1(a)}
		|u_i|^{\frac4{n-4}}|\nabla u_i|^2\,dx=\delta.
		\label{eq:balls}
	\end{equation}
	This follows by varying the radius of the balls: the supremum of the
	corresponding integrals is continuous in the radius and increases from $0$ to $1$.
	In particular, since $B_2$ can be covered by fewer than
	$1/(2\delta)$ unit balls, we have
	\begin{equation}
		\int_{\mathbb R^n\setminus B_2}|u_i|^{\frac4{n-4}}
		|\nabla u_i|^2\,dx\ge\frac12.
		\label{eq:outside}
	\end{equation}
	Passing to a subsequence, we have $u_i\rightharpoonup u$ in
	$\dot H^2(\mathbb R^n)$ and $u_i\to u$ in $H^1_{\mathrm{loc}}(\mathbb R^n)$.
	We shall prove that $u\not\equiv0$.
	
	Assume the contrary. Let $\chi,\eta\in C^\infty(\mathbb R^n)$ be nonnegative,
	$\chi^2+\eta^2=1$, with $\chi=1$ on $B_1$ and $\chi=0$ outside $B_2$.
	Since $u_i\to0$ in $H^1_{\mathrm{loc}}(\mathbb R^n)$,
	\begin{equation}
		\int_{\mathbb R^n}(\Delta u_i)^2\,dx
		=\int_{\mathbb R^n}(\Delta(\chi u_i))^2\,dx
		+\int_{\mathbb R^n}(\Delta(\eta u_i))^2\,dx+o(1).
		\label{eq:energy}
	\end{equation}
	Moreover, let us introduce
	\[
	A_i\coloneqq\int_{\mathbb R^n}|\chi u_i|^{\frac4{n-4}}|\nabla(\chi u_i)|^2\,dx,
	\qquad
	B_i\coloneqq\int_{\mathbb R^n}|\eta u_i|^{\frac4{n-4}}|\nabla(\eta u_i)|^2\,dx
	\]
	and
	\[
	A_i'\coloneqq\int_{\mathbb R^n} \chi^{\frac{2(n-2)}{n-4}}| u_i|^{\frac4{n-4}}|\nabla u_i|^2\,dx,
	\qquad
	B_i'\coloneqq\int_{\mathbb R^n} \eta^{\frac{2(n-2)}{n-4}} | u_i|^{\frac4{n-4}}|\nabla u_i|^2\,dx \,.
	\]
	Then
	\begin{equation}
		A_i = A_i' +o(1),
		\qquad
		B_i = B_i' + o(1).
		\label{eq:cutoff}
	\end{equation}
	Indeed, the error $A_i-A_i'$ in \eqref{eq:cutoff}
	is bounded by a constant times
	\[
	\left(\int_{\mathbb R^n}|u_i|^{\frac4{n-4}}|\nabla u_i|^2\,dx\right)^{1/2}
	\left(\int_{B_2}|u_i|^{\frac{2(n-2)}{n-4}}\,dx\right)^{1/2}
	+\int_{B_2}|u_i|^{\frac{2(n-2)}{n-4}}\,dx,
	\]
	which tends to zero by the compact local Sobolev embedding, and similarly for $B_i-B_i'$.

	We have
	\begin{align}\label{eq:lowerbound}
		1+o(1) \ge (A_i')^{\frac{n-4}{n-2}}+(B_i')^{\frac{n-4}{n-2}}.
	\end{align}	
	Indeed, the definition of $S_{4,2}$, \eqref{eq:energy} and \eqref{eq:cutoff} yield
	\begin{align*}
	1+o(1) & = (S_{4,2})^{-1}\int_{\mathbb R^n}(\Delta u_i)^2\,dx +o(1)
	=(S_{4,2})^{-1} \int_{\mathbb R^n}(\Delta(\chi u_i))^2\,dx 
	+(S_{4,2})^{-1} \int_{\mathbb R^n}(\Delta(\eta u_i))^2\,dx +o(1)\nonumber \\
	& \ge A_i^{\frac{n-4}{n-2}}+B_i^{\frac{n-4}{n-2}} + o(1) \ge (A_i')^{\frac{n-4}{n-2}}+(B_i')^{\frac{n-4}{n-2}} + o(1).
	\end{align*}	
		
	Using $\frac{n-4}{n-2}<1$, it is easy to see that for any $\frac12\geq \sigma>0$ there is an $\varepsilon>0$ such that for all $t\in[\sigma,1-\sigma]$ we have
	\[
	t^\frac{n-4}{n-2} + (1-t)^\frac{n-4}{n-2} \geq (1-\varepsilon)^{-\frac{n-4}{n-2}}.
	\]
	We want to apply this with $\sigma = \delta'/(1+\delta')$ where $\delta'=\min\{\delta,1/2\}$ and $t = A_i'/(A_i'+B_i')$. 
       This is admissible since, by \eqref{eq:normal}, \eqref{eq:balls} and \eqref{eq:outside}, we have $1\ge A_i'\ge\delta$ and $1\ge B_i'\ge\frac12$. We deduce that
	\[
	(A_i')^{\frac{n-4}{n-2}}+(B_i')^{\frac{n-4}{n-2}}
	\ge(1-\varepsilon)^{-\frac{n-4}{n-2}} (A_i'+B_i')^{\frac{n-4}{n-2}}.
	\]
	Combining this with \eqref{eq:lowerbound}, we obtain $(1-\varepsilon)(1+o(1)) \geq A_i'+B_i'$ and therefore, recalling \eqref{eq:normal},
	\begin{equation}
		\liminf_{i\to\infty}\int_{\mathbb R^n}
		\left(1-\chi^q-\eta^q\right)
		|u_i|^{\frac4{n-4}}|\nabla u_i|^2\,dx\ge\varepsilon,
		\label{eq:annulus}
	\end{equation}
	where $q\coloneqq2(n-2)/(n-4)$.
	
	We emphasize that the same $\varepsilon$ works for every such pair of cutoff functions.
	Now for any given integer $K$, choose $K$ pairs $(\chi_k,\eta_k)$ satisfying these conditions so that the supports of $1-\chi_k^q-\eta_k^q$ are contained in disjoint annuli
	in $B_2\setminus B_1$ (for example, we can choose $\chi_k = 1$ in $B_{1+(k-1)/K}$ and $\chi_k = 0$ off $B_{1+k/K}$). Also note that $0\leq 1-\chi_k^q-\eta_k^q\leq 1$. Applying \eqref{eq:annulus} to each such pair, summing with respect to $k$ and noting that
	\[
	0\leq \sum_{k=1}^K \left( 1-\chi_k^q-\eta_k^q \right) \begin{cases} \leq 1 & \text{in}\ B_2,\\
		= 0 & \text{in}\ B_1, \end{cases}
	\]
	we deduce from \eqref{eq:annulus} and \eqref{eq:normal}
	\[
	1 \geq \lim_{i\to\infty} \int_{B_2\setminus B_1} |u_i|^{\frac4{n-4}}|\nabla u_i|^2\,dx
	\geq \liminf_{i\to\infty} \sum_{k=1}^K \int_{\mathbb R^n}
	\left(1-\chi_k^q-\eta_k^q\right) |u_i|^{\frac4{n-4}}|\nabla u_i|^2\,dx\ge K\varepsilon.
	\]
	Since this holds for arbitrarily large $K$, we arrive at a contradiction. Thus
	$u\not\equiv0$.

	\medskip
	
    \noindent\textit{Step 2.} We now prove that $u_i\to u$ strongly in
	$\dot H^2(\mathbb R^n)$. 
	
	After passing to a subsequence we may assume,
	by the local compact embedding, that $u_i\to u$ and
	$\nabla u_i\to\nabla u$ almost everywhere. Put $v_i\coloneqq u_i-u$.
	Weak convergence gives
	\begin{equation}
		S_{4,2}=\int_{\mathbb R^n}(\Delta u)^2\,dx
		+\lim_{i\to\infty}\int_{\mathbb R^n}(\Delta v_i)^2\,dx.
		\label{eq:hilbert}
	\end{equation}
	We claim that
	\begin{equation}
		1=\int_{\mathbb R^n}|u|^{\frac4{n-4}}|\nabla u|^2\,dx
		+\lim_{i\to\infty}\int_{\mathbb R^n}
		|v_i|^{\frac4{n-4}}|\nabla v_i|^2\,dx.
		\label{eq:splitting}
	\end{equation}
	
	Before proving this, let us use it to complete the proof. Combining \eqref{eq:hilbert} and \eqref{eq:splitting} with the definition
	of $S_{4,2}$, we obtain
	\begin{align*}
		S_{4,2} &=\int_{\mathbb R^n}(\Delta u)^2\,dx
		+\lim_{i\to\infty}\int_{\mathbb R^n}(\Delta v_i)^2\,dx \\
		&\ge S_{4,2}\left(\int_{\mathbb R^n}|u|^{\frac4{n-4}}
		|\nabla u|^2\,dx\right)^{\frac{n-4}{n-2}}
		+S_{4,2} \limsup_{i\to\infty} \left(\int_{\mathbb R^n}|v_i|^{\frac4{n-4}}
		|\nabla v_i|^2\,dx\right)^{\frac{n-4}{n-2}} \\
		&= S_{4,2}\left(\int_{\mathbb R^n}|u|^{\frac4{n-4}}
		|\nabla u|^2\,dx\right)^{\frac{n-4}{n-2}}
		+S_{4,2}\left(1-\int_{\mathbb R^n}|u|^{\frac4{n-4}}
		|\nabla u|^2\,dx\right)^{\frac{n-4}{n-2}}
		\ge S_{4,2}.
	\end{align*}
	The last inequality comes from the elementary inequality $t^\frac{n-4}{n-2} + (1-t)^\frac{n-4}{n-2} \geq 1$ for all $t\in[0,1]$. Since $u\not\equiv0$, the first integral is positive (apply the Sobolev inequality to $|u|^{\frac{n-2}{n-4}}$). The fact that the elementary inequality is strict for $t\in(0,1)$ implies that $\int_{\mathbb R^n}|u|^{\frac4{n-4}} |\nabla u|^2\,dx= 1$. Hence $\int_{\mathbb R^n}(\Delta u)^2\,dx=S_{4,2}$ and, by \eqref{eq:hilbert}, $\int_{\mathbb R^n}(\Delta v_i)^2\,dx\to0$.
	This is the claimed strong convergence. Undoing the harmless initial
	multiplication by constants finishes the proof.
	
	Thus, it remains to prove \eqref{eq:splitting}. Expanding $\nabla u_i=\nabla u+\nabla v_i$, the difference between
	the left side and the two terms on the right side of
	\eqref{eq:splitting} (before taking the limit) is
	\begin{align}\label{eq:splittingproof}
		&\int_{\mathbb R^n}|u_i|^{\frac4{n-4}}|\nabla u_i|^2\,dx-\int_{\mathbb R^n}|v_i|^{\frac4{n-4}}|\nabla v_i|^2\,dx-\int_{\mathbb R^n}|u|^{\frac4{n-4}}|\nabla u|^2\,dx\notag\\
        &=\int_{\mathbb R^n}\left(|u_i|^{\frac4{n-4}}-|v_i|^{\frac4{n-4}}\right)
		|\nabla v_i|^2\,dx +2\int_{\mathbb R^n}|u_i|^{\frac4{n-4}}
		\nabla v_i\cdot\nabla u\,dx
		+\int_{\mathbb R^n}\left(|u_i|^{\frac4{n-4}}-|u|^{\frac4{n-4}}\right)
		|\nabla u|^2\,dx.
	\end{align}
	We will show that all three integrals tend to zero.
	
	For the proof we will use the fact that  the Sobolev embeddings give uniform bounds for $u_i,v_i$ in $L^{\frac{2n}{n-4}}(\mathbb R^n)$ and for their gradients in
	$L^{\frac{2n}{n-2}}(\mathbb R^n)$. Using these we now show that
	\begin{equation}
		|u_i|^{\frac4{n-4}}-|v_i|^{\frac4{n-4}}
		\longrightarrow |u|^{\frac4{n-4}}
		\qquad\text{in }L^{\frac n2}(\mathbb R^n).
		\label{eq:powers}
	\end{equation}
	The proof uses the same argument that is behind the proof of the Brezis--Lieb lemma \cite[Theorem 1.9]{Lieb-Loss-Analysis}, but we include the details for the sake of completeness.
	
	For $n\ge8$ (when $4/(n-4)\leq 1$), we use $\bigl||a+b|^{\frac4{n-4}}-|b|^{\frac4{n-4}}\bigr|
	\le|a|^{\frac4{n-4}}$, which implies $\bigl||a+b|^{\frac4{n-4}}-|b|^{\frac4{n-4}}-|a|^{\frac4{n-4}}\bigr|
	\le 2|a|^{\frac4{n-4}}$. Applying this with $a=u$ and $b=v_i$ and using dominated convergence, we arrive at \eqref{eq:powers}.
	
	For $4<n<8$, we use instead
	\(
	\bigl||a+b|^{\frac4{n-4}}-|b|^{\frac4{n-4}}\bigr|
	\le C\left(|a|^{\frac4{n-4}}
	+|a|\,|b|^{\frac{8-n}{n-4}}\right), 
	\)
	which implies that, for any $\varepsilon>0$,
	\[
	\bigl||a+b|^{\frac4{n-4}}-|b|^{\frac4{n-4}}-|a|^{\frac4{n-4}}\bigr|^\frac n2
	\le \varepsilon |b|^\frac{2n}{n-4} + C_\varepsilon |a|^\frac{2n}{n-4}.
	\]
	We abbreviate
	\[
	r_i\coloneqq\left||u_i|^{\frac4{n-4}}-|v_i|^{\frac4{n-4}}-|u|^{\frac4{n-4}}\right|^{\frac n2}.
	\]
	Since $r_i\to0$ almost everywhere, the positive part
	$(r_i-\varepsilon|v_i|^{\frac{2n}{n-4}})_+$ tends to zero and is bounded by
	$C_\varepsilon|u|^{\frac{2n}{n-4}}\in L^1(\mathbb R^n)$.
	Dominated convergence therefore gives
	\[
	\limsup_{i\to\infty}\int_{\mathbb R^n}r_i\,dx
	\le\varepsilon\sup_i\int_{\mathbb R^n}|v_i|^{\frac{2n}{n-4}}\,dx.
	\]
	The supremum is finite, and letting $\varepsilon\to0$ proves
	\eqref{eq:powers} also for $4<n<8$.
	
	Next, note that
	\begin{equation}\label{eq:weakconv}
	|\nabla v_i|^2\rightharpoonup0 \  \text{in}\ L^{\frac n{n-2}}(\mathbb R^n),
	\qquad
	|u_i|^{\frac4{n-4}}\nabla v_i \rightharpoonup 0\ \text{in}\ L^{\frac{2n}{n+2}}(\mathbb R^n),
	\qquad
	|u_i|^{\frac4{n-4}}- |u|^{\frac4{n-4}} \rightharpoonup 0\ \text{in}\ L^{\frac n2}(\mathbb R^n).
	\end{equation}
	All three weak convergences follow from the following simple measure-theoretic result \cite[Exercise 4.16]{BrezisFA}: If $1<p<\infty$ and if $(f_i)$ is a bounded sequence in $L^p$ (of some measure space) that converges a.e.~to zero, then it converges weakly in $L^p$ to zero.
	
	Now the fact that all three integrals in \eqref{eq:splittingproof} tend to zero is an immediate consequence of \eqref{eq:powers} and \eqref{eq:weakconv}. This concludes the proof of \eqref{eq:splitting} and therefore of the theorem.
\end{proof}

%%%%%%%%%%%%%%%%%%%%%%%%
%%%%%%%%%%%%%%%%%%%%%%%%

\section{Classification of minimizers}

In this section we determine all minimizers of $E$.
 We first show that every minimizer is, up to translation and a change of sign, a positive smooth function that is radially symmetric, strictly radially decreasing, and superharmonic (i.e., $-\Delta u\ge0$). In a second step we classify all such minimizers explicitly.

%\subsection{Radial symmetry of minimizers}

We begin with the symmetry and regularity properties.

\begin{lem}\label{lem:radial-minimizer}
	Any minimizer of $E$ is, up to translation and a change of sign, a positive smooth function that is radially symmetric, strictly radially decreasing, and superharmonic.
\end{lem}

\begin{proof}
We split the argument into three steps.

\medskip

\noindent\emph{Step 1: positivity and superharmonicity.}
Let $u\in\dot H^2(\mathbb R^n)$. Let $v\in\dot H^2(\mathbb R^n)$ be the unique solution of
\[
-\Delta v=|\Delta u|\qquad\text{in }\mathbb R^n.
\]
Since $n\ge5$, the Newton kernel is positive, hence $v>0$ and $|u|\le v$ a.e.

Integration by parts yields
\begin{align}
\frac{n}{n-4}\int_{\mathbb R^n}|u|^{\frac4{n-4}}|\nabla u|^2\,dx
&=\int_{\mathbb R^n}(-\Delta u)\,u|u|^{\frac4{n-4}}\,dx\nonumber\\
&\le \int_{\mathbb R^n}|\Delta u|\,|u|^{\frac n{n-4}}\,dx\nonumber\\
&\le \int_{\mathbb R^n}(-\Delta v)\,v^{\frac n{n-4}}\,dx
=\frac{n}{n-4}\int_{\mathbb R^n}v^{\frac4{n-4}}|\nabla v|^2\,dx.
\end{align}
Moreover,
\[
\int_{\mathbb R^n}(\Delta v)^2\,dx=\int_{\mathbb R^n}|\Delta u|^2\,dx=\int_{\mathbb R^n}(\Delta u)^2\,dx.
\]
Consequently, $E(v)\le E(u)$.

Now assume that $u$ is a minimizer. Then $E(v)=E(u)=S_{4,2}$. We claim that $\Delta u$ has constant sign.
Set $f\coloneqq\Delta u$ and $f_{\pm}\coloneqq\max\{\pm f,0\}$. Then
\[
-\Delta(v-u)=2f_+,
\qquad
-\Delta(v+u)=2f_-.
\]
If both $f_+$ and $f_-$ were nontrivial, positivity of the Newton kernel would imply $v-u>0$ and $v+u>0$ a.e., hence $v>|u|$ a.e. Since $|\Delta u|\not\equiv0$, we would have
\[
\int_{\mathbb R^n}|\Delta u|\,|u|^{\frac n{n-4}}\,dx
<\int_{\mathbb R^n}|\Delta u|\,v^{\frac n{n-4}}\,dx
=\int_{\mathbb R^n}(-\Delta v)\,v^{\frac n{n-4}}\,dx,
\]
which forces $E(v)<E(u)$, a contradiction. Thus $\Delta u$ does not change sign.
	If $\Delta u\le 0$, then positivity of the Newton kernel implies $u>0$ almost everywhere; similarly, if $\Delta u\ge 0$, then $-u>0$ almost everywhere. This proves that every minimizer has constant sign.
	
	\medskip
	
	\noindent\emph{Step 2.} We now discuss radial symmetry and monotonicity.
	
	For a measurable function $w$ on $\mathbb{R}^n$ such $|\{ |w|>\lambda\}|<\infty$ for all $\lambda>0$, we denote by $w^*$ its symmetric decreasing rearrangement; see \cite[Section 3]{Lieb-Loss-Analysis}.
	
	We first assume that $u\in\dot H^2(\R^n)$ is arbitrary. Consider the unique solution $w\in\dot H^2(\R^n)$ of
	\[
	-\Delta w=(-\Delta u)^*\qquad\text{in }\mathbb{R}^n.
	\]
	Then $w$ is radial and positive by the positivity of the Newton kernel. Talenti's comparison theorem~\cite{Telenti78} or \cite{Burchard} yields
	\eq{\label{Talenti_eq}
		0\leq u^*\le w.
	}
	Moreover, rearrangement preserves the $L^2$ norm \cite[Section 3.3]{Lieb-Loss-Analysis}, and hence
	\[
	\int_{\mathbb R^n}(-\triangle w)^{2}\,dx
	=\int_{\mathbb R^n}(( -\triangle u)^{*})^{2}\,dx
	=\int_{\mathbb R^n}(-\Delta u)^{2}\,dx.
	\]
	
	By the same integration by parts as in Step 1 and by the Hardy--Littlewood rearrangement inequality \cite[Theorem 3.4]{Lieb-Loss-Analysis},
	\[
	\frac{n}{n-4}\int_{\mathbb{R}^n}|u|^{\frac{4}{n-4}}|\nabla u|^2\,dx
	= \int_{\mathbb R^n} (-\triangle u) u|u|^{\frac{4}{n-4}}\,dx
	\le \int_{\mathbb R^n} |-\triangle u| |u|^{\frac{n}{n-4}}\,dx
	\le \int_{\mathbb R^n} (-\triangle u)^{*} \bigl(|u|^{\frac{n}{n-4}}\bigr)^{*}\,dx.
	\]
	Since $\bigl(|u|^{\frac{n}{n-4}}\bigr)^{*}=(u^*)^{\frac{n}{n-4}}$, we obtain from \eqref{Talenti_eq} that
	\eq{\label{talenti increase energy}
		\int_{\mathbb R^n} (-\triangle u)^{*} \bigl(|u|^{\frac{n}{n-4}}\bigr)^{*}\,dx
		\le \int_{\mathbb R^n} (-\triangle w) w^{\frac{n}{n-4}}\,dx
		= \frac{n}{n-4}\int_{\mathbb{R}^n}w^{\frac{4}{n-4}}|\nabla w|^2\,dx.
	}
	Combining these inequalities, we conclude that $E(w)\le E(u)$. 
	
	Now assume that $u$ minimizes $E$. According to Step 1, after a change of sign, we may assume that $u>0$ and $-\Delta u\geq 0$. The above argument shows that $E(w)=E(u)=S_{4,2}$, hence $w$ is a radial minimizer, with $w>0$ and $-\Delta w\ge0$. In particular,
	\begin{equation}\label{equality}
		\int_{\mathbb R^n} w^{\frac{n}{n-4}}(-\triangle w)\,dx=\int_{\mathbb R^n} (u^*)^{\frac{n}{n-4}}(-\triangle u)^*\,dx.
	\end{equation}
	Elliptic regularity, discussed further in Step 3, yields that $u$ and $w$ are smooth. In what follows, we only need that $u$ and $w$ are continuous and vanish at infinity. These properties imply, in particular, that $\max_{\R^n} u$ and $\max_{\R^n} w$ are attained and finite. The latter is attained at the origin.
	
	By \eqref{Talenti_eq}, we have $w^{\frac{n}{n-4}}(-\triangle w)\ge (u^*)^{\frac{n}{n-4}}(-\triangle u)^*$ on $\R^n$. Since $(-\triangle u)^{*}$ is decreasing and $-\triangle u\not\equiv0$, there is an $r_0>0$ such that $(-\triangle u)^{*}>0$ in $B_{r_0}$. This, together with \eqref{equality}, implies that $w=u^*$ on $B_{r_0}$. In particular, we have
	\begin{equation}\label{eq:equalmax}
	\max_{\R^n} w = w(0)=u^*(0)=\max_{\R^n} u.
	\end{equation}
	Let $x_0\in\mathbb R^n$ such that $\max_{\R^n} u=u(x_0)$. By the Newton representation, the monotonicity of $y\mapsto |y|^{2-n}$ and the Hardy--Littlewood inequality \cite[Theorem 3.4]{Lieb-Loss-Analysis} we find
	\begin{align*}
		\max_{\R^n} u=u(x_0)
		&=C_n\int_{\mathbb R^n}|x_0-y|^{2-n}(-\Delta u)(y)\,dy
		=C_n\int_{\mathbb R^n}|y|^{2-n}(-\Delta u)(y+x_0)\,dy\\
		&\le C_n\int_{\mathbb R^n}|y|^{2-n}(-\Delta u)^*(y)\,dy
		=w(0)=\max_{\R^n} w.
	\end{align*}
	By \eqref{eq:equalmax}, the inequality here is an equality. Since $y\mapsto |y|^{2-n}$ is strictly symmetric-decreasing, \cite[Theorem~3.4]{Lieb-Loss-Analysis} implies
	$-\Delta u(y+x_0)=(-\Delta u)^*(y)$. In view of
	\[
	u(x) = C_n\int_{\mathbb R^n}|y|^{2-n}(-\Delta u)(y+x)\,dy = C_n\int_{\mathbb R^n}|y|^{2-n}(-\Delta u)^*(y+x-x_0)\,dy = w^*(x-x_0),
	\]
	the function $u$ is radially symmetric with respect to $x_0$.
	
	Moreover,
	$r^{n-1}w'(r)=\int_0^{r}(\Delta w)s^{n-1}ds=-\int_0^{r}(-\Delta u)^*s^{n-1}ds<0$
	for $r>0$, since $-\Delta u\not\equiv0$. Hence $w$ (and therefore $u$) is strictly radially decreasing.
	
	\medskip
	
	\noindent\emph{Step 3.} Finally, let us discuss the regularity of minimizers.
	
	Let $\phi:\mathbb R^n\cup\{\infty\}\to\mathbb S^n$ be the stereographic projection and let $N\coloneqq\phi(\infty)$. For a Euclidean minimizer $u$, define
	\[
	\widetilde u(\phi(x))\coloneqq\left(\frac{1+|x|^2}{2}\right)^{\frac{n-4}{2}}u(x).
	\]
	By conformal covariance, $\widetilde u\in H^2(\mathbb S^n)$ and $\widetilde u$ is a minimizer of the corresponding spherical functional, and satisfies the corresponding Euler--Lagrange equation. The regularity theorem for minimizers in \cite[Appendix~A]{GeWangWei} then implies $\widetilde u\in C^\infty(\mathbb S^n)$. Since
	\[
	u(x)=\left(\frac{2}{1+|x|^2}\right)^{\frac{n-4}{2}}\widetilde u(\phi(x)),
	\]
	we obtain that $u\in C^\infty(\R^n)$ together with the decay estimates
	\[
	|\nabla^ju(x)|=O(|x|^{4-j-n})\,\,\,\text{for}\,\,j=0,1,2,\quad \text{as}\,\,|x|\to\infty.
	\]
	This concludes the proof.
\end{proof}

\begin{remark}
	The argument at the beginning of Step 2 implies that
	\[
	S_{4,2} = \inf_{u\in \dot H^2(\mathbb{R}^n)\setminus\{0\}\, \mathrm{radial}} E(u) \,.
	\]
	This observation can be used to give an alternative, elementary proof of the existence of an optimizer. This is most easily carried out in logarithmic variables, as in Section \ref{sec:classificationode}, where it becomes clear that the functional is subcritical. We emphasize, however, that our argument in Section \ref{sec:existence} gives more than just the existence of an optimizer: it gives the relative compactness (up to symmetries) of optimizing sequences. The latter, stronger result is crucial in many applications. 
\end{remark}

%%%%%%%%%%%%%%%%%%%%%%%%%%%%%%%%

%\subsection{Classification of minimizers}

We now classify all positive minimizers.

\begin{thm}\label{Thm A}
Any positive minimizer of $E$ is of the form
\[
 u(x)=\left(\frac{a}{d+|x-\bar{x}|^{2}}\right)^{\frac{n-4}{2}},
\]
for some $a>0$, $d>0$, and $\bar{x}\in\mathbb{R}^n$.
\end{thm}

By Lemma~\ref{lem:radial-minimizer}, every minimizer is radially symmetric after a translation and hence satisfies a fourth-order radial ODE. This is a possible route to proving Theorem \ref{Thm A}, which we will pursue in Section \ref{sec:classificationode}. Here we rather  follow an approach inspired by Chen--C.~Li--Ou~\cite{ChenLiOu}, based on Kelvin transforms and a calculus lemma of Y.~Li--Zhu~\cite{LiZhu}.

Let $u$ be a positive minimizer of $E$. Then $u$ satisfies the Euler--Lagrange equation
\begin{equation}\label{eq:el}
\Delta^{2}u=\Lambda \, u^{\frac{2}{n-4}}\bigl(-\Delta u^{\frac{n-2}{n-4}}\bigr)\quad\text{in }\mathbb{R}^{n}
\end{equation}
for some constant $\Lambda>0$. Moreover, by Step 3 in the proof of Lemma \ref{lem:radial-minimizer} the limit
\eq{\label{alpha}
\alpha\coloneqq\lim_{|x|\to\infty}u(x)|x|^{n-4}
}
exists and is finite.

We claim that $\alpha>0$. Indeed, if $\alpha=0$, then the Kelvin transform
\[
 w(x)\coloneqq\frac{1}{|x|^{n-4}}u\!\left(\frac{x}{|x|^{2}}\right)
\]
satisfies $\lim_{x\to0}w(x)=0$ and $\lim_{|x|\to\infty}w(x)=0$. By conformal invariance of $E$, the function $w$ is also a minimizer. By Lemma~\ref{lem:radial-minimizer} it is (after translation) radial and strictly decreasing, and it satisfies $-\Delta w\ge0$. This contradicts $\lim_{x\to0}w(x)=0$.

We now prove the following lemma, inspired by Chen--C.~Li--Ou~\cite{ChenLiOu}, which encodes a Kelvin-type invariance property of minimizers.
\begin{lem}\label{lem:Kelvin keeps}
For any $a\in\mathbb{R}^{n}$ there exists $\lambda(a)>0$ such that $\lambda(a)^{n-4}u(a)=\alpha$ and
\[
 u(\lambda(a)x+a)=\frac{1}{|x|^{n-4}}u\!\left(\frac{\lambda(a)x}{|x|^{2}}+a\right)
 \qquad\text{for all }x\in\mathbb{R}^{n}\setminus\{0\}.
\]
Equivalently,
\[
 u(y)=\left(\frac{\lambda(a)}{|y-a|}\right)^{n-4}u\!\left(a+\frac{\lambda(a)^{2}(y-a)}{|y-a|^{2}}\right)
 \qquad\text{for all }y\in\mathbb{R}^{n}\setminus\{a\}.
\]
\end{lem}

\begin{proof}
Let $a\in\mathbb{R}^n$. Set
\[
\lambda(a)\coloneqq\left(\frac{\alpha}{u(a)}\right)^{\!\frac{1}{n-4}},
\]
where $\alpha$ is given by \eqref{alpha}. For any unit vector $e\in\mathbb{S}^{n-1}$ define
\[
 v(x)\coloneqq\frac{\lambda(a)^{\frac{n-4}{2}}}{|x|^{n-4}}u\!\left(\lambda(a)\left(\frac{x}{|x|^{2}}-e\right)+a\right).
\]
By conformal invariance, $v$ is also a minimizer. By Lemma~\ref{lem:radial-minimizer}, the function $v$ is radial (with respect to its maximum point) and strictly radially decreasing.
Since $\alpha=\lim_{|x|\to\infty}u(x)|x|^{n-4}$ and $\lambda(a)^{n-4}u(a)=\alpha$, it is easy to see
\[
\lim_{x\to0}v(x)=\alpha\lambda(a)^{-\frac{n-4}{2}},
\qquad
v(e)=\alpha\lambda(a)^{-\frac{n-4}{2}}.
\]
Since $v$ is radial and strictly decreasing away from its maximum point, it is symmetric with respect to the hyperplane through $\tfrac12 e$ with normal $e$. Hence, for any $-\tfrac12<h<\tfrac12$,
\[
 v\bigl((\tfrac12-h)e\bigr)=v\bigl((\tfrac12+h)e\bigr).
\]
This implies
\[
\frac{\lambda(a)^{\frac{n-4}{2}}}{\left|\frac12-h\right|^{n-4}}u\!\left(\lambda(a)\frac{\frac12+h}{\frac12-h}e+a\right)
=
\frac{\lambda(a)^{\frac{n-4}{2}}}{\left|\frac12+h\right|^{n-4}}u\!\left(\lambda(a)\frac{\frac12-h}{\frac12+h}e+a\right).
\]

Given $x\neq0$, write $x=te$ with $t=|x|$ and $e=x/|x|$. Choose $h$ so that $t=\frac{\frac12+h}{\frac12-h}$. Then the previous identity yields
\[
 u(\lambda(a)te+a)=\frac{1}{t^{n-4}}u\!\left(\lambda(a)\frac{1}{t}e+a\right),
\]
that is,
\[
 u(\lambda(a)x+a)=\frac{1}{|x|^{n-4}}u\!\left(\lambda(a)\frac{x}{|x|^{2}}+a\right).
\]
The equivalent formulation follows by setting $y=\lambda(a)x+a$.
\end{proof}
We recall the following lemma of Li--Zhu~\cite{LiZhu}.

\begin{lem}\label{lem:Li-Zhu Duke 1995}
Let $f\in C^{1}(\mathbb{R}^{n})$, $n\ge1$, and $\nu>0$. Suppose that for every $x\in\mathbb{R}^{n}$ there exists $\lambda(x)>0$ such that
\[
\left(\frac{\lambda(x)}{|y-x|}\right)^{\nu}f\!\left(x+\frac{\lambda(x)^{2}(y-x)}{|y-x|^{2}}\right)=f(y),
\qquad y\in\mathbb{R}^{n}\setminus\{x\}.
\]
Then there exist $a\ge0$, $d>0$, and $\bar{x}\in\mathbb{R}^{n}$ such that
\[
 f(x)=\pm\left(\frac{a}{d+|x-\bar{x}|^{2}}\right)^{\frac{\nu}{2}}.
\]
\end{lem}

\begin{proof}[Proof of Theorem~\ref{Thm A}]
Lemma~\ref{lem:Kelvin keeps} shows that $u$ satisfies the hypothesis of Lemma~\ref{lem:Li-Zhu Duke 1995} with $\nu=n-4$. Therefore, for some $\bar x\in\mathbb{R}^n$ and some $a>0$, $d>0$,
\[
 u(x)=\left(\frac{a}{d+|x-\bar{x}|^{2}}\right)^{\frac{n-4}{2}}.
\]
This proves Theorem~\ref{Thm A}.
\end{proof}

We remark that only $C^0$ regularity is needed in Lemma~\ref{lem:Li-Zhu Duke 1995}; see the appendix by Li and Nirenberg in~\cite{Li2004}. A version for finite measures is established in~\cite{Frank-Lieb2010}.

%%%%%%%%%%%%%%%%%%%%%%%%%%
%%%%%%%%%%%%%%%%%%%%%%%%%%

\section{An alternative classification argument}\label{sec:classificationode}

In this section we give an alternative proof of Theorem \ref{Thm A} using ODE analysis, which is of independent interest.

Let $u\in\dot H^2(\R^n)$ be a positive, radial solution of the Euler--Lagrange equation \eqref{eq:el} with some $\Lambda>0$. We write $u(r)$ instead of $u(x)$ with $|x|=r$ and pass to logarithmic coordinates $t=\log r$, setting
\[
u(r)=r^{-\frac{n-4}{2}}v(\log r) \,.
\]
Then the equation for $u$ becomes the equation 
\begin{equation}
	L_{\frac{n-4}{2}}L_{\frac n2}v
	=\Lambda v^{\frac2{n-4}}
	L_{\frac{n-2}{2}}\left(v^{\frac{n-2}{n-4}}\right),
	\label{eq:autonomous-Lambda}
\end{equation}
where we write
\[
\begin{aligned}
	L_{\frac{n-4}{2}} =-\frac{d^2}{dt^2}+\frac{(n-4)^2}{4},\qquad
	L_{\frac n2}=-\frac{d^2}{dt^2}+\frac{n^2}{4},\qquad
	L_{\frac{n-2}{2}}=-\frac{d^2}{dt^2}+\frac{(n-2)^2}{4}.
\end{aligned}
\]
A positive homoclinic is a solution \(v\in C^4(\R)\) satisfying
\begin{equation}\label{homoclinic limits}
    v>0,\qquad v,v',v'',v'''\longrightarrow0
\quad\text{as }t\to\pm\infty.
\end{equation}
These conditions are satisfied, in particular, for $v$ coming from a minimizer $u$, as follows from Step 3 of the proof of Lemma \ref{lem:radial-minimizer}.

\begin{thm}
	\label{thm:main}
	Let $\Lambda>0$. If $v$ is a positive homoclinic for \eqref{eq:autonomous-Lambda}, then, for some $t_0\in\R$,
	\[
	{
		v(t)=\left(\frac{(n-4)(n+2)}{4\Lambda}\right)^{\frac{n-4}{4}}
		\sech^{\frac{n-4}{2}}(t-t_0).
	}
	\]
\end{thm}

We begin with a preliminary reduction. Replacing \(v\) by
\(\bigl(4\Lambda/((n-4)(n+2))\bigr)^{(n-4)/4}v\), we therefore reduce to
\begin{equation}
	{
		L_{\frac{n-4}{2}}L_{\frac n2}v
		=\frac{(n-4)(n+2)}4\,v^{\frac2{n-4}}
		L_{\frac{n-2}{2}}\left(v^{\frac{n-2}{n-4}}\right).
	}
	\label{eq:normalized}
\end{equation}
Direct differentiation shows that
\begin{equation}
	v_*(t)\coloneqq\sech^{\frac{n-4}{2}}t
	\label{eq:bubble}
\end{equation}
solves \eqref{eq:normalized}.  We shall prove that every positive
homoclinic of \eqref{eq:normalized} is a translate of \(v_*\).

\medskip

The following proof uses some ideas from the fourth-order classification arguments in
\cite{FrankKonig19,FrankKonig19b}. Other parts, including identities invlving the perfect square in Lemma \ref{lem:factorization} and the inequality in Lemma \ref{lem:trapping} are specific to the equation that we study here.

We outline the proof in four steps:
\begin{itemize}
\item[1.] Factor the fourth-order operator into second-order factors with perfect-square identities, yielding the essential \eqref{eq:logarithmic-bound}, see Lemma \ref{lem:factorization};
\item[2.] Use the factors to prove the positivity of the ratio of the radial sectional curvature to the sectional curvature of
a plane tangent to a spherical orbit,  see Remark \ref{geometric meaning of eta} and  Lemma \ref{lem:convex};
\item[3.] Extract an ordering of homoclinic profiles, see Lemma \ref{lem:ordering};
\item[4.] Use self-adjointness and a signed integral identity to rule out a
strictly ordered pair, see Lemma \ref{lem:no-ordering}.
\end{itemize}

The principal new feature here is that the nonlinearity contains the
second-order expression \(L_{\frac{n-2}{2}}\left(v^{\frac{n-2}{n-4}}\right)\).  The modified factors \(T_\pm\) in
\eqref{eq:Tpm-def} below  absorb this term and produce the perfect squares
\eqref{eq:Tminus}--\eqref{eq:Tplus}.  This also yields
\(L_{\frac{n-2}{2}}\left(v^{\frac{n-2}{n-4}}\right)>0\), i.e., the scalar curvature of $v^{\frac 4{n-4} }(dt^2+g_{\mathbb S^{n-1}})$ is positive, as a consequence rather than as a hypothesis.  The ordering
step is obtained from a simple monotonicity property of the zero-energy
reduced system (Lemma~\ref{lem:trapping}).  This property requires only
the leading-order endpoint behaviour of the homoclinic.

%%%%%%%%%%%%%%%%%%%%%%%%%%%%%%%%

We now start to prove Theorem \ref{thm:main}. Define
\[
A_-\coloneqq-\frac{d^2}{dt^2}-2\frac{d}{dt}+\frac{n(n-4)}4,
\qquad
A_+\coloneqq-\frac{d^2}{dt^2}+2\frac{d}{dt}+\frac{n(n-4)}4,
\]
so that
\begin{equation}
	A_+A_-=A_-A_+=L_{\frac{n-4}{2}}L_{\frac n2}.
	\label{eq:linear-factor}
\end{equation}
Set
\begin{equation}
	T_\pm\coloneqq A_\pm v
	-\frac{(n-4)(n-2)(n+2)}{4n}\,v^{\frac n{n-4}}.
	\label{eq:Tpm-def}
\end{equation}

\begin{lem}
	\label{lem:factorization}
	Every positive homoclinic $v$ of \eqref{eq:normalized} satisfies
	\begin{align}
		A_+T_-&=\frac{(n-2)(n+2)}{2(n-4)}
		v^{\frac{8-n}{n-4}}\left(v'-\frac{n-4}{2}v\right)^2,
		\label{eq:Tminus}\\
		A_-T_+&=\frac{(n-2)(n+2)}{2(n-4)}
		v^{\frac{8-n}{n-4}}\left(v'+\frac{n-4}{2}v\right)^2.
		\label{eq:Tplus}
	\end{align}
	Moreover, there holds
	\begin{equation}
		T_\pm>0,\qquad A_\pm v>0,
		\qquad
		-\frac{n-4}{2}<\frac{v'}v<\frac{n-4}{2}.
		\label{eq:logarithmic-bound}
	\end{equation}
\end{lem}

\begin{proof}

By computations, we have
\begin{equation}
\begin{aligned}
A_+\!\left(v^{\frac n{n-4}}\right)
={}&\frac n{n-2}v^{\frac2{n-4}}
L_{\frac{n-2}{2}}\!\left(v^{\frac{n-2}{n-4}}\right)-\frac{2n}{(n-4)^2}v^{\frac{8-n}{n-4}}
\left(v'-\frac{n-4}{2}v\right)^2.
\end{aligned}
\label{eq:power-factor}
\end{equation}
Applying \(A_+\) to \eqref{eq:Tpm-def} and using
\eqref{eq:linear-factor}, we get
\[
A_+T_-=L_{\frac{n-4}{2}}L_{\frac n2}v
-\frac{(n-4)(n-2)(n+2)}{4n}
A_+\!\left(v^{\frac n{n-4}}\right).
\]
Substituting \eqref{eq:normalized} and \eqref{eq:power-factor} into the above formula, we have \eqref{eq:Tminus}.
	  Reflection \(t\mapsto-t\)
	gives \eqref{eq:Tplus}.
	
	Since \(T_\pm\to0\) at both ends, a negative value of \(T_-\) would yield
	a negative global minimum.  At that point,
	\[
	A_+T_-=-T_-''+\frac{n(n-4)}4T_-<0,
	\]
	contrary to \eqref{eq:Tminus}.  Thus \(T_-\ge0\), and likewise
	\(T_+\ge0\).  Formula \eqref{eq:Tpm-def} now yields \(A_\pm v>0\).
	
	The identities
	\[
	\begin{aligned}
		\left[e^{-\frac n2t}
		\left(v'+\frac{n-4}{2}v\right)\right]'
		&=-e^{-\frac n2t}A_+v<0,\\
		\left[e^{\frac n2t}
		\left(v'-\frac{n-4}{2}v\right)\right]'
		&=-e^{\frac n2t}A_-v<0
	\end{aligned}
	\]
	and the zero limits at \(+\infty\) and \(-\infty\), respectively, give
	\(v'+\frac{n-4}{2}v>0\) and \(v'-\frac{n-4}{2}v<0\).
	Hence $-\frac{n-4}{2}<\frac{v'}v<\frac{n-4}{2}$ holds, and hence the right-hand sides of
	\eqref{eq:Tminus}--\eqref{eq:Tplus} are positive everywhere.

Finally, if \(T_-(t_1)=0\) at some point $t_1$, the nonnegativity of
\(T_-\) yields \(T_-'(t_1)=0\) and \(T_-''(t_1)\ge0\).  Thus the
left-hand side of \eqref{eq:Tminus} at \(t_1\) is
\[
(A_+T_-)(t_1)
=-T_-''(t_1)+2T_-'(t_1)+\frac{n(n-4)}4T_-(t_1)
=-T_-''(t_1)\le0,
\]
contradicting the strictly positive right-hand side of \eqref{eq:Tminus}.
Likewise, if \(T_+(t_1)=0\), then \(T_+'(t_1)=0\) and
\(T_+''(t_1)\ge0\), so the left-hand side of \eqref{eq:Tplus} is
\[
(A_-T_+)(t_1)
=-T_+''(t_1)-2T_+'(t_1)+\frac{n(n-4)}4T_+(t_1)
=-T_+''(t_1)\le0.
\]
This contradicts \eqref{eq:Tplus}.
Hence \(T_\pm>0\).
\end{proof}

%%%%%%%%%%%%%%%%%%%%%%%%%%%%%%%%
The identities \eqref{eq:Tminus}--\eqref{eq:Tplus} are the main analogue,
for the present equation, of the second-order factorization device used in
the fourth-order classification arguments of Frank--K\"onig \cite{FrankKonig19,FrankKonig19b}.

\medskip

Now we rewrite the metric in an alternative conformal factor.  
From now on, let
\begin{equation}
	f\coloneqq v^{-\frac2{n-4}},\qquad
	\eta\coloneqq\frac{ff''-(f')^2}{f^2-(f')^2}.
	\label{eq:f-def}
\end{equation}
Thus the conformal metric is
\[
g=v^{\frac4{n-4}}(dt^2+g_{\mathbb S^{n-1}})
=f^{-2}(dt^2+g_{\mathbb S^{n-1}}).
\]
The homoclinic limits and \eqref{eq:logarithmic-bound} give
\begin{equation}
f(t)\longrightarrow\infty\quad(t\to\pm\infty),
	\qquad |f'|<f.
	\label{eq:f-bound}
\end{equation}
\begin{remark}\label{geometric meaning of eta}
Let \(\bar e_1,\ldots,\bar e_{n-1}\) be a local orthonormal frame for
\(g_{\mathbb S^{n-1}}\).  Then
\[
e_0=f\partial_t,\qquad e_i=f\bar e_i\quad(1\le i\le n-1)
\]
is an orthonormal frame for \(g\).  
Thus \(\eta\) is the ratio of the radial sectional curvature to the
sectional curvature of a plane tangent to a spherical orbit, both
computed in the ambient metric \(g\).  In terms of the curvature
operator, \(ff''-(f')^2\) and \(f^2-(f')^2\) are its eigenvalues on
the subspaces spanned by \(e_0\wedge e_i\) and \(e_i\wedge e_j\),
respectively.  Since \(f^2-(f')^2>0\), positivity of \(\eta\) is
equivalent to positivity of all sectional curvatures of \(g\).
For the standard profile \(f=\cosh t\), both eigenvalues equal one,
so \(\eta\equiv1\) and \(g\) is the unit round sphere metric with its
two poles removed.
\end{remark}

The standard profile \(v_*\) corresponds to \(f_*(t)\coloneqq\cosh t\).
The definitions imply
\begin{equation}
	\eta-1=\frac{f(f''-f)}{f^2-(f')^2},
	\qquad
	\left(\artanh\frac{f'}f\right)'=\eta,
	\label{eq:eta-f}
\end{equation}
and
\begin{equation}
	\bigl(f^2-(f')^2\bigr)'=-2f'(f''-f).
	\label{eq:f-square-derivative}
\end{equation}

Define the Hamiltonian energy in terms of \(f\) by
\begin{equation}
	\begin{aligned}
		H\coloneqq f^{-n}\Biggl[&f^2(f''-f)^2-2f^2f'(f'''-f')
	+2f\bigl(f^2+(n-2)(f')^2\bigr)(f''-f)\\
		&+\frac{n^2-4}{4}\bigl(f^2-(f')^2\bigr)
		\bigl(1-f^2+(f')^2\bigr)\Biggr].
	\end{aligned}
	\label{eq:f-Hamiltonian}
\end{equation}
Recall \(v=f^{-(n-4)/2}\). Differentiating \eqref{eq:f-Hamiltonian}, we have
\[
H'=-\frac8{(n-4)^2}v'\left[
L_{\frac{n-4}{2}}L_{\frac n2}v
-\frac{(n-4)(n+2)}4v^{\frac2{n-4}}
L_{\frac{n-2}{2}}\left(v^{\frac{n-2}{n-4}}\right)\right]=0.
\]
To determine this constant $H$, we differentiate \(v=f^{-(n-4)/2}\)
successively.  By the homoclinic limits and \(|f'|<f\), it holds that
\[
f^{-\frac{n-2}{2}}f^{(j)}\longrightarrow0
\quad(j=1,2,3),\qquad
f^{-\frac{n-4}{2}}\longrightarrow0
\quad(t\to\pm\infty).
\]
Together with \(0<f^2-(f')^2\le f^2\), these limits imply that every term of
\eqref{eq:f-Hamiltonian} tends to zero.  Hence \(H=0\).
Rearranging the identity \(H=0\), we have
\begin{equation}
	{\begin{aligned}
			2f'f^{n-2}\left(\frac{f''-f}{f^{n-2}}\right)'
			={}&(f''-f)(f''+f)
			+\frac{n^2-4}{4f^2}\bigl(f^2-(f')^2\bigr)
			\bigl(1-f^2+(f')^2\bigr).
	\end{aligned}}
	\label{eq:f-energy}
\end{equation}

For later use, we mention that \(T_\pm\) satisfy
\begin{equation}
	\begin{aligned}
		\frac{2T_\pm}{(n-4)v}
		={}&\left(\eta+\frac{n-4}{2}\right)
		\left(1-\frac{(f')^2}{f^2}\right)
		+2\mp2\frac{f'}f-\frac{n^2-4}{2nf^2}>0,
	\end{aligned}
	\label{eq:T-f}
\end{equation}
which is obtained directly from \eqref{eq:Tpm-def} and Lemma \ref{lem:factorization}.

%%%%%%%%%%%%%%%%%%%%%%%%%%%

\begin{lem}
	\label{lem:center}
	The function \(f\) has exactly one critical point, which is a strict
	minimum.  After translation we may assume
	\begin{equation}\label{zero}
	    f'(0)=0,\qquad f'(t)<0\quad \text{for all}\ t<0,\qquad f'(t)>0\quad \text{for all}\ t>0.
	\end{equation}
\end{lem}

\begin{proof}
	At a point where \(f'=0\), from \eqref{eq:f-energy} and \eqref{eq:T-f}, it holds that
	\[
	(f'')^2=\frac{n^2}{4}f^2-\frac{n^2-4}{4},
	\qquad
	\frac{f''}f+\frac n2>\frac{n^2-4}{2nf^2}>0.
	\]
	Substituting the above equality into the inequality yields
	\[
	\frac{f''}f\left(\frac{f''}f+\frac n2\right)>0.
	\]
	Thus \(f''>0\) at the critical point of $f$, and then  every critical point is a strict local minimum.
	Since \(f\to\infty\) at both ends, a minimum exists; two distinct
	minima would have a local maximum between them.  Hence the critical
	point is unique, and we finish the proof.
\end{proof}

Without loss of generality, \emph{ in the following proof we only consider the function $f$ satisfying \eqref{zero}.}
Now we want to prove the positivity of $\eta$. 

\begin{lem}
	\label{lem:convex}
	Every positive homoclinic satisfies
	\[
	ff''-(f')^2>0,\qquad f''>0
	\quad\text{on }\R.
	\]
	Equivalently, there holds \(\eta>0\).
\end{lem}

\begin{proof}
	By Lemma~\ref{lem:center},  \(\eta(0)>0\).
	Writing \eqref{eq:f-energy} in terms of \(\eta\), on \((0,\infty)\), we have
	\begin{equation}
		\begin{aligned}
			\eta'=\frac{f}{2f'}\Biggl[&
			\left(1+3\frac{(f')^2}{f^2}\right)\eta^2
			+2(n-4)\frac{(f')^2}{f^2}\eta
		+\frac{(n-4)^2}{4}\frac{(f')^2}{f^2}
			-\frac{n^2}{4}+\frac{n^2-4}{4f^2}\Biggr].
		\end{aligned}
		\label{eq:eta-f-equation}
	\end{equation}
	At a zero of \(\eta\) in \((0, +\infty)\), positivity of \(T_+\) in
	\eqref{eq:T-f} implies
	\begin{equation}
\frac{n^2-4}{4f^2}
<\frac{n(n-4)}4\left(1-\frac{(f')^2}{f^2}\right)
 +n\left(1-\frac{f'}f\right),
\label{eq:eta-zero-bound}
\end{equation}
	and using this in \eqref{eq:eta-f-equation}, we obtain
	\[
	\begin{aligned}
		\frac{2f'}f\eta'
		&=\frac{(n-4)^2}{4}\frac{(f')^2}{f^2}
		-\frac{n^2}{4}+\frac{n^2-4}{4f^2}\\
		&<-\frac{2f'}f\left(\frac{n-4}{2}\frac{f'}f+\frac n2\right)<0.
	\end{aligned}
	\]
	In particular, \(\eta'<0\) at every zero of \(\eta\) with \(t>0\).

Suppose, for contradiction, that $\eta\le 0$ at some $t>0$. Then there exists $t_1>0$ with $\eta(t_1)<0$.
We claim that $\eta(t)<0$ for $t>t_1$. Otherwise, by continuity there would be a first time $t_2>t_1$ with
\[
\eta(t_2)=0,\qquad \eta(t)<0\quad \text{for}\,\,t_1\le t<t_2.
\]
In particular $\eta'(t_2)\ge 0$, which contradicts the fact that $\eta'(t_2)<0$ at any zero of $\eta$ with $t_2>0$.

Since $f(t)\to\infty$ as $t\to\infty$, we can choose $T\ge t_1$ large such that
\[
\frac{n^2-4}{4f(t)^2}\le n-4\qquad\text{for all }t\in[T,\infty).
\]
	Whenever \(-1\le\eta\le0\),  due to \(0<f'/f<1\), we have
\[
\left(1+3\frac{(f')^2}{f^2}\right)\eta^2
+2(n-4)\frac{(f')^2}{f^2}\eta\\
+\frac{(n-4)^2}{4}\frac{(f')^2}{f^2}
-\frac{n^2}{4}+\frac{n^2-4}{4f^2}\le -(n-4)<0.
\]
	Since \(f/(2f')>1/2\), \eqref{eq:eta-f-equation} yields \(\eta'\le-(n-4)/2\).
    
	Therefore, there exists a large constant $T_1$ such that
	\(\eta\le-1\) for $t\ge T_1$.  But \eqref{eq:eta-f} would then imply
	\[
\left(\artanh\frac{f'}f\right)'=\eta\le-1,
	\]
	whereas \(0<f'/f<1\) on \((0, +\infty)\), contradicting the strict positivity of
\(\artanh(f'(t)/f(t))\).
	Hence \(\eta>0\) for \(t\ge0\).
	For the reflected profile \(\widetilde f(t)=f(-t)\), the corresponding
	quantity is \(\widetilde\eta(t)=\eta(-t)\).  The same argument therefore applies
	on \((-\infty, 0)\).  Finally, since 
	\(ff''=\eta\bigl(f^2-(f')^2\bigr)+(f')^2>0\), we get \(f''>0\).
\end{proof}

\begin{cor}
	\label{cor:nonlinear-positive}
	Every positive homoclinic satisfies
	\begin{equation}
		\begin{aligned}
			\frac{L_{\frac{n-2}{2}}(v^{\frac{n-2}{n-4}})}
			{v^{\frac{n-2}{n-4}}}
			&=\frac{n-2}{4f^2}
			\left[2\bigl(ff''-(f')^2\bigr)
			+(n-2)\bigl(f^2-(f')^2\bigr)\right]>0.
		\end{aligned}
		\label{eq:nonlinear-positive}
	\end{equation}
\end{cor}

\begin{proof}
	
	The claimed positivity follows from \(f^2-(f')^2>0\) and
	Lemma~\ref{lem:convex}.
\end{proof}

%%%%%%%%%%%%%%%%%%%%%%%%%%%%%%%%

By Lemma~\ref{lem:convex}, \(f''>0\).  On \((0, +\infty)\),
\eqref{eq:f-energy} and \eqref{eq:f-square-derivative} imply the following trapping
property.

\begin{lem}
	\label{lem:trapping}
	Unless \(f(t)=\cosh t\) on \(\R\),
	\[
	\bigl(f(t)^2-f'(t)^2-1\bigr)\bigl(f''(t)-f(t)\bigr)>0
	\qquad\text{for all}\ t\ne0.
	\]
\end{lem}

\begin{proof}
We first work on $(0, +\infty)$, where the preceding lemmas give
\[
f'>0,\qquad f''>0,\qquad (f')^2<f^2,\qquad
f(t)\longrightarrow\infty\quad(t\to\infty).
\]
We divide this proof into four claims.
Claims~1 and~2 exclude the two ways in which two factors $f(t)^2-f'(t)^2-1$ and $f''(t)-f(t)$  can have opposite signs.  Claims~3 and~4 then handle their
possible zeros.

\medskip
\noindent\textbf{Claim 1.}
\emph{There does not exist \(t_1>0\) which satisfies both \(f(t_1)^2-(f'(t_1))^2>1\) and \(f''(t_1)<f(t_1)\).}

\smallskip
%\noindent\textit{Proof of Claim 1.}
Suppose the set of the points satisfying $f^2-(f')^2>1$ and $f''<f$ is not empty.  In this set,
\eqref{eq:f-energy} and \eqref{eq:f-square-derivative} give
\begin{equation}\label{monotone1}
\left(\frac{f''-f}{f^{n-2}}\right)'<0,
\qquad \bigl(f^2-(f')^2\bigr)'>0,
\end{equation}
because \(f'>0\) and \(f''+f>0\).  Take a point $t_1$ in this set. For \(t> t_1\), as long as the solution remains
in this set, \eqref{monotone1} yields that
\[
\begin{aligned}
f(t)^2-f'(t)^2
&> f(t_1)^2-f'(t_1)^2>1,\\
\frac{f''(t)-f(t)}{f(t)^{n-2}}
&<\frac{f''(t_1)-f(t_1)}{f(t_1)^{n-2}}<0.
\end{aligned}
\]
So the solution would always remain inside this set.
In particular, for all such \(t\), we have
\[
\frac{f''(t)}{f(t)}
\le 1+
\frac{f''(t_1)-f(t_1)}{f(t_1)^{n-2}}\,f(t)^{n-3}
\longrightarrow-\infty,
\]
since \(f(t)\to\infty\).  This contradicts \(f''>0\), and we prove Claim~1.

\medskip
\noindent\textbf{Claim 2.}
\emph{There does not exist \(t_1>0\) which satisfies both \(f^2-(f')^2<1\) and \(f''>f\).}

\smallskip
%\noindent\textit{Proof of Claim 2.}
Suppose the set of points satisfying $f^2-(f')^2<1$ and $f''>f$ is not empty, and this set contains a point \(t_1>0\).  As Claim 1, in this set
\eqref{eq:f-energy} and \eqref{eq:f-square-derivative} yield
\[
\left(\frac{f''-f}{f^{n-2}}\right)'>0,
\qquad \bigl(f^2-(f')^2\bigr)'<0.
\]
Consequently, for \(t\ge t_1\), as long as the solution remains in this
set, we have
\[
\begin{aligned}
\frac{f''(t)-f(t)}{f(t)^{n-2}}
&\ge\frac{f''(t_1)-f(t_1)}{f(t_1)^{n-2}}>0,\\
f(t)^2-f'(t)^2
&\le f(t_1)^2-f'(t_1)^2<1.
\end{aligned}
\] 
The same argument as in Claim~1 therefore shows that
the set conditions and these bounds hold for every \(t\ge t_1\).
In particular, there holds that for every \(t\ge t_1\)
\[
f''(t)\ge
\frac{f''(t_1)-f(t_1)}{f(t_1)^{n-2}}\,f(t)^{n-2}.
\]
 Multiplying by \(2f'(t)>0\) and integrating
from \(t_1\) to \(t\),  we get
\[
\begin{aligned}
(f'(t))^2\ge{}&(f'(t_1))^2+\frac{2\bigl(f''(t_1)-f(t_1)\bigr)}
{(n-1)f(t_1)^{n-2}}
\bigl(f(t)^{n-1}-f(t_1)^{n-1}\bigr).
\end{aligned}
\]
  Dividing the above inequality by
\(f(t)^2\), we infer
\((f'(t))^2/f(t)^2\to\infty\) due to  \(f(t)\to\infty\) and $\frac{2(f''(t_1)-f(t_1))}
{(n-1)f(t_1)^{n-2}}>0$.  This contradicts
\((f')^2<f^2\). Thus, we prove Claim~2.

\medskip
\noindent\textbf{Claim 3.}
\emph{There does not exist \(t>0\) satisfies exactly one of \(f^2-(f')^2=1\) and \(f''=f\).}

\smallskip
%\noindent\textit{Proof of Claim 3.}
First suppose \(f''(t_1)=f(t_1)\) and
\(f(t_1)^2-f'(t_1)^2\ne1\) at some \(t_1>0\). 
By \eqref{eq:f-energy}, we have
\[
\left(\frac{f''-f}{f^{n-2}}\right)'
=\frac{n^2-4}{8f'f^n}
\bigl(f^2-(f')^2\bigr)\bigl(1-f^2+(f')^2\bigr),
\qquad t=t_1.
\]
Recall that $f'(t)>0$ for $t>0$. If
\(f(t_1)^2-f'(t_1)^2>1\), then $\left(\frac{f''-f}{f^{n-2}}\right)'$ is negative.  Therefore,
for \(t>t_1\) sufficiently close to \(t_1\), we have
\(f^2-(f')^2>1\) and \(f''<f\), contradicting Claim~1.
If \(f(t_1)^2-f'(t_1)^2<1\), the derivative is positive, so all $t>t_1$ close to $t_1$ satisfy \(f^2-(f')^2<1\) and \(f''>f\),
contradicting Claim~2.

Next suppose \(f(t_1)^2-f'(t_1)^2=1\) and
\(f''(t_1)\ne f(t_1)\).  By \eqref{eq:f-square-derivative},
\[
\bigl(f^2-(f')^2\bigr)'(t_1)
=-2f'(t_1)\bigl(f''(t_1)-f(t_1)\bigr).
\]
If \(f''(t_1)>f(t_1)\), then $(f^2-(f')^2)'(t_1)$ is negative, and   \(f^2-(f')^2<1\) and \(f''>f\) near
\(t>t_1\), contrary to
Claim~2.  If \(f''(t_1)<f(t_1)\), then $(f^2-(f')^2)'(t_1)$ is positive, and
 \(f^2-(f')^2>1\) and \(f''<f\) near  \(t>t_1\),
contrary to Claim~1.  This proves Claim~3.

\medskip
\noindent\textbf{Claim 4.}
\emph{If both factors vanish at a time \(t_1>0\), then
\(f(t)=\cosh t\) on \(\R\).}

\smallskip
%\noindent\textit{Proof of Claim 4.}
Suppose
\[
f(t_1)^2-f'(t_1)^2=1,\qquad f''(t_1)=f(t_1).
\]
Substituting these equalities into \eqref{eq:f-energy}, we infer
\[
2f'(t_1)\bigl(f'''(t_1)-f'(t_1)\bigr)=0.
\]
Since \(f'(t_1)>0\), we obtain \(f'''(t_1)=f'(t_1)\).
Because
\(f(t_1)^2-f'(t_1)^2=1\), the following positive function
\[
g(t)\coloneqq f(t_1)\cosh(t-t_1)+f'(t_1)\sinh(t-t_1)=\cosh \left(t-t_1+s\right),
\]
where $s=\operatorname{arsinh} f^{\prime}\left(t_1\right)$,  and $g(t)$ agrees with \(f\) through the third
derivative at \(t_1\).   Thus, by the uniqueness, we know that \(f=\cosh\left(t-t_1+s\right)\)  on
\(\R\), and then  \(f(t)=\cosh t\) due to  $f'(0)=0$.  This proves Claim~4.

For $(-\infty, 0)$, we reflect $f(t)=f(-t)$, and then complete the proof.
\end{proof}

%%%%%%%%%%%%%%%%%%%%%%%%%%%%%%%%

Now we give a proof to distinguish the solution by different values of $f$ at zero.

\begin{lem}
	\label{lem:ordering}
	Under \eqref{zero},
	\[
	\begin{array}{lll}
		f(0)=1&\Longrightarrow&v=v_*\quad\hbox{on }\R,\\
		f(0)<1&\Longrightarrow&v>v_*\quad\hbox{on }\R,\\
		f(0)>1&\Longrightarrow&v<v_*\quad\hbox{on }\R.
	\end{array}
	\]
\end{lem}

\begin{proof}
	Suppose \(f\ne\cosh\).  Lemma~\ref{lem:trapping} and
	\eqref{eq:f-square-derivative} give
	\[
	\begin{aligned}
		\left[\bigl(f^2-(f')^2-1\bigr)^2\right]'
		&=-4f'\bigl(f^2-(f')^2-1\bigr)(f''-f) <0\qquad\text{for all}\ t>0.
	\end{aligned}
	\]
	If \(f(0)=1\), then \(f(0)^2-f'(0)^2-1=0\), so $(f^2-(f')^2-1)^2$ would
	decrease strictly from zero.  This is impossible, proving the first
	assertion.
	
	If \(f(0)\ne1\), Lemma~\ref{lem:trapping} shows that
	\(f^2-(f')^2-1\) cannot
	vanish on either half-line, and \[
\sgn\bigl(f^2-(f')^2-1\bigr)
=\sgn(f''-f)
=\sgn\bigl(f(0)^2-1\bigr),\qquad t\ne0.
\]
	For \(t>0\), the elementary identity
	\[
	f(t)-f(0)\cosh t
	=\int_0^t\sinh(t-s)\bigl(f''(s)-f(s)\bigr)\,ds
	\]
	thus yields
	\[
	\begin{array}{lll}
		f(0)<1&\Longrightarrow&f(t)<f(0)\cosh t<\cosh t,\\
		f(0)>1&\Longrightarrow&f(t)>f(0)\cosh t>\cosh t.
	\end{array}
	\]
	The same inequalities hold for \(t<0\) by reflection, and their strict
	versions at zero follow directly from \(f(0)\ne1\).  Since
	\(v=f^{-(n-4)/2}\) and \(v_*=\sech^{(n-4)/2}t\), we complete the proof.
\end{proof}

%%%%%%%%%%%%%%%%%%%%%%%

We now  are ready to finish the proof by proving the following lemma.
\begin{lem}
	\label{lem:no-ordering}
	Two positive homoclinic solutions \(v,w\) of \eqref{eq:normalized}
	cannot satisfy \(v>w\) on \(\R\).
\end{lem}

\begin{proof}
	For \(T>0\), subtracting the equations tested against \(w\) and \(v\), by integration by parts on \([-T,T]\) and expansion of the powers, we  have
    \begin{equation}\label{eq:picone-integrated}
	\begin{aligned}
		&\int_{-T}^{T}\Biggl[
		\frac{2(n-2)}{(n-4)^2}\,
		v^{\frac n{n-4}}w
		\left(\frac{v'}v-\frac{w'}w\right)^2
		+\left(v^{\frac n{n-4}}w-vw^{\frac n{n-4}}\right)
		\frac{L_{\frac{n-2}{2}}\!\left(w^{\frac{n-2}{n-4}}\right)}
		{w^{\frac{n-2}{n-4}}}
		\Biggr]dt\\
		&=
		\frac{4}{(n-4)(n+2)}
		\Biggl[
		wv'''-w'v''+w''v'-w'''v
		-\frac{n^2-4n+8}{2}(wv'-w'v)
		\Biggr]_{-T}^{T}\\
		&\qquad+
		\frac{n-2}{n-4}
		\left[
		w v^{\frac4{n-4}}v'-v^{\frac n{n-4}}w'
		\right]_{-T}^{T}=o(1)\quad \text{as}\,\, T\rightarrow +\infty,
	\end{aligned}
    \end{equation}
	where we have used homoclinic limits \eqref{homoclinic limits}.

	Suppose now that \(v>w\) everywhere. 
By Corollary~\ref{cor:nonlinear-positive} and
\(
v^{\frac n{n-4}}w-vw^{\frac n{n-4}}=vw\left(v^{\frac4{n-4}}-w^{\frac4{n-4}}\right)>0,
\)
the left-hand side of \eqref{eq:picone-integrated} is strictly positive with a positive lower bound. This contradicts \eqref{eq:picone-integrated} as $T\to\infty$.

Similarly, interchanging $v$ and $w$ in \eqref{eq:picone-integrated} excludes the case $v<w$ everywhere.
\end{proof}

\begin{proof}[Proof of Theorem~\ref{thm:main}]
	Let \(v\) be a normalized positive homoclinic,
	and center \(f\coloneqq v^{-2/(n-4)}\) by Lemma~\ref{lem:center}.
	If \(f(0)=1\), then \(v=v_*\) by Lemma~\ref{lem:ordering}.
	Otherwise that lemma gives either \(v>v_*\) or \(v<v_*\) on
	\(\R\), contradicting to Lemma~\ref{lem:no-ordering}.
	Thus every normalized positive homoclinic is a translate of \(v_*\).	
\end{proof}

\medskip

{\em On the use of AI:} We used ChatGPT to assist with language editing and presentation. All mathematical arguments and proofs in the final manuscript were checked and written by the authors.

\bibliography{bibSobolev}
\bibliographystyle{amsplain}
\end{document}